\documentclass[12pt,a4paper,twoside]{amsart}
\usepackage{amsmath}
\usepackage{amssymb}
\usepackage{amsfonts}
\usepackage{a4}
\usepackage{color}

\numberwithin{equation}{section}

\theoremstyle{plain}

\newtheorem{theorem}[subsection]{Theorem}
\newtheorem{definition}[subsection]{Definition}
\newtheorem{proposition}[subsection]{Proposition}
\newtheorem{lemma}[subsection]{Lemma}
\newtheorem{remark}[subsection]{Remark}
\newtheorem{question}[subsection]{Question}
\newtheorem{corollary}[subsection]{Corollary}

\theoremstyle{definition}

\newcommand{\mC}{{\mathbb C}}

\newcommand{\mE}{{\mathbb E}}
\newcommand{\mF}{\mathbb F}

\newcommand{\mN}{\mathbb N}

\newcommand{\mT}{\mathbb T}

\newcommand{\mZ}{{\mathbb Z}}

\newcommand{\Gg}{\gamma}
\newcommand{\GG}{\Gamma}

\newcommand{\ep}{\epsilon}

\newcommand{\mcA}{\mathcal A}
\newcommand{\mcB}{\mathcal B}

\newcommand{\mcF}{\mathcal F}

\newcommand{\ti}{\tilde}

 \newcommand\End{\operatorname{End}}
  \newcommand\Ker{\operatorname{Ker}}
   \newcommand\Hom{\operatorname{Hom}}

\begin{document}
\title{Approximate cohomology}
\begin{abstract}
Let $k$ be a field,  $G$ be an abelian group and $r\in \mN.$
 Let $L$ be an infinite dimensional $k$-vector space. For any $m\in \End_k(L)$ we denote by $r(m)\in [0,\infty ]$ the rank of $m$. We define by $R(G,r,k)\in [0,\infty ]$ the minimal $R$ such that for any map  $A:G \to  \End_k(L)$ with $r(A(g'+g'')-A(g')-A(g''))\leq r$, $g',g''\in G$ there exists a homomorphism $\chi :G\to \End_k(L)$ such that $r(A(g)-\chi (g))\leq R(G, r, k)$ for all $g\in G$. 

We show the finiteness of $R(G,r,k)$ for the case when $k$ is a finite field, $G=V$ is a $k$-vector space $V$ of countable dimension. We actually prove a generalization of this result.

 In addition 
we introduce a notion of {\it Approximate Cohomology}  groups $H^k_\mcF (V,M)$  (which is a purely algebraic analogue of the notion of  
$\ep$-representation (\cite{ep})) and interperate our result as a computation of 
the group $H^1_\mcF (V,M)$ for some $V$-modules $M$. 
\end{abstract}

\author{David Kazhdan}
\address{Einstein Institute of Mathematics,
Edmond J. Safra Campus, Givaat Ram 
The Hebrew University of Jerusalem,
Jerusalem, 91904, Israel}
\email{david.kazhdan@mail.huji.ac.il}

\author{Tamar Ziegler}
\address{Einstein Institute of Mathematics,
Edmond J. Safra Campus, Givaat Ram 
The Hebrew University of Jerusalem,
Jerusalem, 91904, Israel}
\email{tamarz@math.huji.ac.il}

\thanks{The second author is supported by ERC grant ErgComNum 682150}

\maketitle

\section{Introduction}

Let $k$ be a field,  $G$ be an abelian group and $r\in \mN.$
 Let $L$ be an infinite dimensional $k$-vector space, $End(L)=End_k(L)$ and $M\subset \End(L)$ be 
the subspace of operators of finite rank. For any $m\in \End(L)$ we denote by $r(m)\in [0,\infty ]$ the rank of $m$. We define by $R(G,r,k)\in [0,\infty ]$(correspondingly $R^f(G,r,k))$ the minimal number $R$ such that for any map  $A:G\to \End(L)$(correspondingly a map $A:G\to M)$, $g\to m_g$ with 
$r(A(g'+g'')-A(g')-A(g''))\leq r$, $g',g''\in G$ there exists a homomorphism $\chi :G\to \End(L)$ such that $r(A(g)-\chi (g))\leq R(G,r,k)$ for all $g\in G$.

It is easy to see that in the case when $G=\mZ$ and $\operatorname{char} (k)\neq 2$ we have $R(\mZ,1, \mF)\leq 2$. We sketch the proof:  one studies the rank 
$\le 1$ operators $r_{m,n}=A(m+n)-A(m)-A(n)$. Since $r_{0,0}$ of rank $\le 1$, we replace $r_{m,n}$ by $r_{n,m}-r_{0,0}$ and can then assume that $r_{0,0}=0$. Under this condition one must show that $r_{n,m}$ is a coboundary of rank one operators.   The operators $r_{n,m}$ satisfy the equation $r_{a,-a}=r_{a+c, -a}+r_{a,c}$. From this one can deduce that either there is a subspace of codimension $1$ in the kernel of all three operators, or a subspace of  dimension $1$ containing the image of all three. One shows inductively that this property holds for all operators $r_{m,n}$. Unfortunately we don't know whether
$R(\mZ ,2,\mC)<\infty$.

In this paper we first show that $R^f(V,r,k)<\infty$ in the case when $k=\mF_p$ and 
$G$ is a $k$- vector space $V$ of countable dimension and then 
show that  $R(V,r,k)=R^f(V,r,k)$. 

Actually we prove the analogous bound in a more general case when $M$ is replaced by
 the space of tensors $L^1\otimes L^2\otimes ...\otimes L^n$. To simplify the exposition
 we assume that $L_i=L^\vee$ and that $Im(A)$ is contained in the subset $Sym^d(L^\vee)\subset {L^\vee}^{\otimes n}$ of
 symmetric tensors. In other words we consider map $A:V\to M^d$ where $M^d$ is isomorphic to the space of homogeneous polynomials of degree
 $d$ on $L$. We denote by $N^d\subset M^d$ the subspace of multilinear polynomials.

Now some formal definitions.

\begin{definition}[Filtration]
Let $M$ be an abelian group.  A {\em filtration} $\mcF$ on $M$ is an increasing 
sequence of subsets $M_i, 1\leq i\leq \infty$ of $M$ such that for each $i,j$ there exists $c(i,j)$ such that $M_i+M_j\subset M_{c(i,j)}$ and $M=\cup _iM_i$.
Two filtrations $M_i,M_i'$ on $M$ are equivalent if there exist
functions $a,b: \mathbb Z _+\to  \mathbb Z _+$ such that $M_i\subset M_{a(i)}'$ and $M'_i\subset M_{b(i)}$.
\end{definition}

\begin{definition} [Algebraic rank filtration] Let $k$ be a field. Fix $d\geq 2$  and consider the
$k$-vector space $ M=M^d$  of homogeneous polynomials $P$ of
degree $d$ in variables $x_j$, $j\geq 1$.
For a non-zero homogeneous  polynomial  $P$ on a $k$-vector space $W$, $P\in k[W^\vee ]$  of degree $d\geq 2$ we
define the {\it rank} $r(P)$ of $P$ as $r$ ,where $r$ is the minimal number $r$ such that it is possible to write $P$ in the form
$$P = \sum ^ r_{i=1} l_iR_i,$$ 
where $l_i, R_i\in \bar k[W^\vee ]$ are homogeneous polynomials of positive degrees (in \cite{schmidt} this is called the $h$-invariant).
We denote by $\mcA _d$ the  filtration on $M$ such that  $M_n$ is the subset of polynomials $P$ with $r(P) <  n$.
For the $k$-space $\mathcal P^d$ of non-homogeneous polynomials  of degree $\le d$ define the rank similarly, and denote by $\mcB_d$ the corresponding filtration. \\
\end{definition}

\begin{definition}[Finite rank homomorphisms] 
Let $V$ be a countable vector space over $k$. We say that a linear map $P:V\to M^d$ is of {\em finite rank} if we can write $P$ as a sum $\sum _{k=1}^{d-1}P_k$ where each $P_k$ is a finite sum $P_k=\sum _jQ_jR_j$ where $R_j\in M^k$ and $Q_j$ is a linear map from $V$ to $M^{d-k}$.  Denote  $\Hom _f(V,M^d)$ the subspace of $\Hom (V,M^d) $ of finite rank maps. 
\end{definition}

Now we can formulate our main result. 
\begin{theorem}\label{Main} For any finite field $k=\mF _p ,d<p-1, r\geq 0$ there exists $R=R(r,k,d)$ such that for any map  
$$A:V\to M=M^d, \quad r(A(v'+v'')-A(v')-A(v''))\leq r, \quad v',v''\in V$$ there exists a homomorphism 
$\chi_A : V\to M$ such that $r(A(v)-\chi_A (v))\leq R$. Moreover the homomorphism $\chi _A$ is unique up to an addition of a homomorphism of a finite rank.
\end{theorem}
\begin{question} $a)$  Is there a bound on $R$ independent of $k$ ? Moreover, does there exist $c(d)$ such that $R(V,r,k,d)\leq c(d)r$?

$b)$ Could we drop the condition $d<p-1$ if $Im(A)\subset N^d$?
\end{question}

\begin{remark} Theorem \ref{Main} does not hold for $p = d= 2$, see \cite{tao-example} for a function from $\mF_2^n$ to the space of quadratic forms over $\mF_2$ such that $f(u+v)-f(u)-f(v)$ is of rank $\le 3$ but for $n$ sufficiently large $f$ does not differ from a linear function by a function taking values in bounded rank quadratics.  In the low characteristic case the same proof shows that obstructions come from  {\em non-classical polynomials} see Remark \ref{nonclassical}.  In the case when $G$ is a finite cyclic group, $k$ any field, $d=2$, one can show that  $C(1,k)\le 2$. 
\end{remark}

We can reformulate Theorem \ref{Main} as an example of a computation of {\it approximate cohomology } groups.

\begin{definition}[Approximate cohomology]
\begin{enumerate}
\item 
Let $M$ be an abelian group, and let  $\mcF=\{M_i\}$ be a filtration on $M$.
\item   Let $G$ be a discrete group acting on  $M$ preserving the subsets $M_i$.
A cochain $r:G^n\to M$ is an {\em approximate $n$-cocycle} if  $Im
(\partial r)\subset M_i$ for some $i\in \mZ _+$.  It is clear that the set $Z^n_\mcF$ of approximate $n$-cocycles is a subgroup of the group $C^n$ of $n$-chains which
 depends only on the equivalence class of a filtration $\mcF$.
\item  A cochain $r:G^n\to M$ is an {\em approximate $n$-coboundary} if there exists an $n-1$-cochain $t\in C^{n-1}$ such that 
$Im (r-\partial t)\subset M_i$ for some $ i\in \mZ _+$. 
 It is clear that the set $B^n_\mcF$ of approximate $n$-coboundaries is a subgroup of $ \tilde Z ^n_\mcF$.
\item We define $H^n_\mcF=Z^n_\mcF /B^n_\mcF$.
\item Since any cocycle is an approximate cocycle and any coboundary is an approximate coboundary we have a morphism 
$a^n_\mcF :H^n(G,M)\to H^n_\mcF (G,M)$. 
\end{enumerate}
\end{definition}
In this paper we consider the case when the group $V$ acts trivially on $M$. So the group $Z^1(V,M)$ of $1$-cocycles coincides with the group $\Hom(V,M)$ of linear maps, the subgroup of coboundaries $B^1(V,M)\subset Z^1(V,M)$ is equal to $\{ 0\}$
and therefore $H^1(V,M)=\Hom(V,M)$. In this case we can reformulate the Theorem \ref{Main} in terms of  a computation of the map $a^1_\mcF$.
\begin{corollary}\label{trivial-action-V} Let $k$ be a  prime finite field of characteristic $p$, $V$ be a countable vector space over $k$ acting trivially on $(M, \mcF)=(M^d,\mcA _d)$
and assume that $p>d+1$.  Then the map $a^1: H^1(V,M)=\Hom(V,M)\to \tilde H_\mcF ^1(V,M)$ is surjective, and  $\Ker(a^1) = \Hom _f(V,M_d)$. 
\end{corollary}
\begin{question} How to describe $H_\mcF ^n(V,M)$ for $n>1$ ?
\end{question}
\begin{corollary}\label{translation-action-V} Let $k$ be a prime  finite field of characteristic $p$, and let $V$ be a countable vector space over $k$.  
Consider the filtration $(\mathcal P^d, \mcB_d)$ with $W=V$ and with $V$-acting by translations $(v .P)(x) = P(x+v)$ and assume that $p>d+1$.   Then the map $a^1_\mcB$ is surjective.
\end{corollary}

\begin{proof} Let $P: V \to  \mathcal P^d$ where $\mathcal P^d$  is the space of polynomial of degree $\le d$. . 
We assume
\[
\partial P(v,v')(x)= P(v+v')(x)-P(v)(x+v')-P(v')(x)
\]
is of rank $\le i$ for any $v,v' \in V$. Let $Q(v)$ be the homogeneous degree $d$ term of $P(v)$.  Then since 
$P(v)(x+v')-P(v)(x)$ is of degree $<d$  we have 
\[
Q(v+v')(x)-Q(v)(x)-Q(v')(x)
\]
is  of rank $\le i+1$. 

\end{proof}

\begin{remark}  The proof of  Theorem \ref{Main} uses the inverse theorem for the Gowers norms \cite{btz,tz-inverse}. One can use also prove the reverse implication modifying the arguments in \cite{sam}, and
  thus an independent proof of Theorem \ref{Main} could  lead to a new proof of the inverse conjecture for the Gowers norms.
\end{remark}

\section{Proof of Theorem \ref{Main}}
 For a function $f$ on a finite set $X$ we define
\[
\mE_{x \in X}f(x) = \frac{1}{|X|}\sum_{x \in X}f(x).
\]
We use $X \ll_L Y$ to denote the estimate 
$|X| \le C(L) |Y|$, where the constant $C$ depends only on $L$.  We fix a prime  finite  field $k$ of order $p$ and degree $d<p-1$ and suppress the dependence of all bounds on $k,d$. We also fix a non-trivial additive character $\psi$  on $k$.      \\  
\ \\
 For a function $f:G \to H$ a function between abelian groups we denote
$\Delta_h f(x) = f(x+h) - f(x)$. if $f:G_1 \times G_2 \to H$ then for $g_1 \in G_1$ we write $\Delta_{g_1} f$ shorthand for $\Delta_{(g_1,0)}f$. \\
\ \\
Let $V$ be a finite vector space over $k$.  Let $F: V \to k$. The $m$-th {\em Gowers norm}  of $F$  is defined by
\[
\|\psi(F)\|_{U_m}^{2^m} = \mE_{v, v_1, \ldots, v_m \in V} \psi(\Delta_{v_m} \ldots  \Delta_{v_1} F(v)).
\]
These were introduced by Gowers in \cite{gowers}, and were shown to be norms for $m>1$.\\
\ \\
For a homogeneous polynomial $P$ on $V$ of degree $d$ we define
\begin{equation}\label{multilinear}
\tilde P(x_1, \ldots, x_d) = \Delta_{x_d}\ldots \Delta_{x_1}P(x).
\end{equation}
This is a multilinear homogeneous form in $x_1, \ldots, x_d \in V$ such  that 
 $$P(x) = \frac{1}{d!}\tilde P(x,\ldots, x).$$

\begin{proposition}\label{main-finite}
There exists a function $R(L)$ such that for any finite dimensional $k$-vector space $V$ and a map $P : V \to M= M_d$ such that 
$r(P(v+v') -P(v) -P(v')) \le  L$ for $v,v' \in V$ there exists a linear map $Q : V \to M$ such that $r(P( v) - Q(v))\le  R(L)$.
\end{proposition}

\begin{proof}
The proof is based on an argument of \cite{gt-equivalence}, \cite{lovett}. Our aim is to show that if $\partial P(v,v')(x)$ is of rank $\le L$ for all $v,v' \in V$ then there exists a homogeneous polynomial  $Q(v)(x)$ of degree $\le d$ such that $\partial P(v,v')=0$ and $P(v)-Q(v) \ll_{L} 1$.

\begin{lemma}
Let $P:V^d \to k$ be a multilinear homogeneous polynomial of degree $d \ge 2$ and rank $L$. Then $\mE_{\bar x \in V} \psi(P(x))\ge  C_{L,d}$ for some 
positive constant $C_{L,d}$ depending only on $L,d$. 
\end{lemma}

\begin{proof}
We prove this by induction on $d$.  For quadratics:   $P(x_1,x_2)= \sum_{i=1}^L l^1_i(x_1)l^2_i(x_2)$. If $x_1 \in \bigcap_i \Ker(l_i^1)$ then  $ \psi(P({\bar x})) \equiv 1$, thus on a subspace $W$ of codimension at most $L$ we have $ \psi(P({\bar x})) \equiv 1$, so that 
\[
\sum_{x_1 \in W} \mE_{x_2 \in V }\psi(P({\bar x}))=|W|.
\]
If $x_1$ is outside $W$ then the inner sum is nonnegative.  \\
\ \\
Suppose now that $d>2$. Let $P(x_1,\ldots, x_d)$ be multilinear homogeneous polynomial of degree $d$ and rank $L$. 
Write
\[
P(x_1, \ldots, x_d) = \sum_{j \le d} \sum_{i\le L_j} l_i^j(x_j)Q_i^j(x_1, \ldots, \hat x_j, \ldots, x_d) + \sum_{k\le M} T_k({\bar x})R_k({\bar x})
\]
with ${\bar x} = (x_1, \ldots, x_d)$, $l_i^j$ are linear  and $T_k({\bar x})R_k({\bar x})$ is homogenous multilinear in $x_1, \ldots x_d$ such that the degrees of $T_k({\bar x})$ and $R_k({\bar x})$ are  $\ge 2$, and $\sum_j L_j+M=L$. \\
 \\
If for all $j$ we have  $L_j=0$, then $P({\bar x})= \sum_{k\le L} T_k({\bar x})R_k({\bar x})$ with  the degree of $T_k({\bar x}),R_k({\bar x}) \ge 2$.   For $x_1 \in V$ write $P_{x_1}(x_2, \ldots, x_d) =  P(x_1,x_2, \ldots, x_d)$.  Then $P_{x_1}$ is of rank $L$ and degree $d-1$ for all $x_1$ and we obtain the claim by induction. \\
\ \\
Otherwise there is a $j$, such that $L_j>0$, without loss of generality $j=1$.
Let $W = \bigcap_{i} \Ker(l_i^1)$, then $W$ is of codimension at most $L_1$. For   $x_1 \in W$ let  $P_{x_1}(x_2, \ldots, x_d) =  P(x_1,x_2, \ldots, x_d)$.  Consider the sum 
\[
 \mE_{x_2, \ldots, x_d \in V}  \psi(P_{x_1}(x_2, \ldots, x_d)).
\]
For $x_1 \in W$ we have $P_{x_1}$ is of rank $\le L-L_1$ and homogeneous of degree $d-1$. 
By the induction hypothesis the above sum  is $\ge C_{L, d-1}$, so that 
\[
\sum_{x_1 \in W} \mE_{x_2, \ldots, x_d \in V}  \psi(P_{x_1}(x_2, \ldots, x_d)) \ge C_{L, d-1} |W|. 
\]
For any $x_1 \notin W$,  $P_{x_1}$ is of degree $d-1$, and of rank $<\infty$ thus by the induction hypothesis,
\[
 \mE_{x_2, \ldots, x_d \in V}  \psi(P_{x_1}(x_2, \ldots, x_d)) \ge 0.
\]
Thus 
\[
\mE_{x_1 \in V} \mE_{x_2, \ldots, x_d \in V}  \psi(P_{x_1}(x_2, \ldots, x_d)) \ge C_{L, d-1} |W|/|V|,
\]
and we obtain the claim. 
\end{proof}
\ \\
Let  $P:V\to M_d$ be a map such that 
 for all $u,v$:
\[
rk(P(v)+P(u) - P(v+u))<L.
\] 
We define a function on $V\times W^d$ by
 $$f(v,x_1, \ldots, x_d) = \psi(\tilde P(v)(x_1, \ldots, x_d)) =  \psi(\tilde  P(v)({\bar x})).$$

\begin{lemma} $\|f\|_{U^{d+2}} \ge c_{L}$. 
\end{lemma}

\begin{proof}
We expand
\[
\Delta_{(v_{d+2},{\bar h_{d+2}})} \ldots  \Delta_{(v_1,{\bar h_1})}\tilde P(v)({\bar x}) = \sum_{k=0}^{d+2} \Delta_{{\bar h_{d+2}}} \ldots  \Delta_{{\bar h_{k+1}}}( \Delta_{v_k} \ldots  \Delta_{v_1}(\tilde P(v)))({\bar x}+{\bar h_1 + \ldots + h_k})
\]
with $v_i \in V$ and ${\bar h_i} \in V^d$.  Since $P_v$ is of degree $d$ the above is equal
\[
\sum_{k=2}^{d+2} \Delta_{{\bar h_{d+2}}} \ldots  \Delta_{{\bar h_{k+1}}}( \Delta_{v_k} \ldots  \Delta_{v_1} \tilde P(v))({\bar x}+{\bar h_1 + \ldots +\bar  h_k})
\]
Since $rk(P(v)+P(u) - P(v+u))<L$, for any $v_1+u_1=v_2+u_2$  we have
\[
rk(P(v_1)+P(u_1) - P(v_2)-P(u_2) )< 2L.
\]
that for $k\ge 2$ we have  $\Delta_{v_k} \ldots  \Delta_{v_1}\tilde P(v)$ is of rank $\ll_L 1$. For fixed $v, v_1, \ldots, v_k$, 
the above polynomial can be expresses as 
 a multilinear  homogeneous polynomial of degree $d$  in  $y_1, \ldots, y_d$ with 
\[
y_j= (x_j,h_1^j,  \ldots, h_{d+2}^j).
\]
which is of  rank that is bounded in terms of $L, d$. Now apply previous lemma.
\end{proof}

By the inverse theorem for the Gowers norm \cite{btz,tz-inverse} there is a polynomial $Q$ on $V^{d+1}$ of degree $d+1$ on with  $Q:V \times V^d \to k$ s.t.
such that 
\[
|\mE_{v,x_1, \ldots, x_d }\psi(\tilde P(v)(x_1, \ldots, x_d)-Q(v, x_1, \ldots,x_d)) | \gg_L 1. 
\]
By an application of the triangle and Cauchy-Schwarz inequalities we obtain
\[\begin{aligned}
&|\mE_{v,x_1, \ldots, x_d }\psi(\tilde P(v)(x_1, \ldots, x_d)-Q(v, x_1, \ldots,x_d)) |^2 \\
& \le \left|\mE_{v,x_1, \ldots,x_{d-1}}\left| \mE_{x_d }\psi(\tilde P(v)(x_1, \ldots, x_d)-Q(v, x_1, \ldots,x_d))\right| \right|^2\\
&\le \mE_{v,x_1, \ldots,x_{d-1}}\left| \mE_{x_d }\psi(\tilde P(v)(x_1, \ldots, x_d)-Q(v, x_1, \ldots,x_d)) \right|^2\\
&= \mE_{v,x_1, \ldots,x_{d-1}} \mE_{x_d,x'_d }\psi(\tilde P(v)(x_1, \ldots, x_d+x'_d) -\tilde P(v)(x_1, \ldots, x_d)-Q(v, x_1, \ldots,x_d+x_d')+Q(v, x_1, \ldots,x_d)) \\
&= \mE_{v,x_1, \ldots,x_{d-1}} \mE_{x_d,x'_d }\psi(\tilde P(v)(x_1, \ldots, x_{d-1},x'_d)-\Delta_{x_d'}Q(v, x_1, \ldots,x_d)).
\end{aligned}\]
Where the last equality  follows  from the fact that $\tilde P$ is homogeneous multilinear form 
and thus 
\[
\tilde P(v)(x_1, \ldots, x_d+x'_d) -\tilde P(v)(x_1, \ldots, x_d)= \tilde P(v)(x_1, \ldots, x_{d-1},x'_d).
\]
Applying Cauchy-Schwarz  $d-1$ more  times we obtain
\[
 \mE_{v,x_1, x_1', \ldots, x_d,x'_d }\psi(\tilde P_v(x'_1, \ldots, x'_d)-\Delta_{x_1'}\ldots \Delta_{x_d'}Q(v, x_1, \ldots,x_d))  \gg_{L} 1
\]
One more application of Cauchy-Schwarz  gives
\[
 \mE_{v,v',x_1, x_1', \ldots x_d,x'_d }\psi((\tilde P(v+v')-\tilde P(v))(x'_1, \ldots, x'_d)-\Delta_{v'}\Delta_{x_1'}\ldots \Delta_{x_d'}Q(v, x_1, \ldots,x_d))  \gg_{L} 1.
\]
 Since  $Q$ is a polynomial of degree $d+1$,  $\Delta_{v'}\Delta_{x_1'}\ldots \Delta_{x_d'}Q$ is independent of $v,x_1, \ldots, x_d$
so we obtain
\[
|\mE_{v,v',x'_1, \ldots, x'_d }\psi((\tilde P(v+v')-\tilde P(v))(x'_1, \ldots, x'_d)-\tilde Q(v', x'_1, \ldots,x'_d)) | \gg_{L} 1
\]
with $\tilde Q$ a multilinear homogeneous form in $v', x'_1, \ldots,x'_d$. Denote
 by 
 $\tilde Q(v)$ the function on $W^d$ given by $\tilde Q(v, \bar x)=Q(v,\bar x)$. Then
\[
\mE_{v,v'}|\mE_{x'_1, \ldots, x'_d }\psi((\tilde P(v+v')-\tilde P(v)- \tilde Q(v'))(x'_1, \ldots, x'_d) | \gg_{L} 1.
\]
By \cite{gt-polynomial} (Proposition 6.1)  and \cite{BL} (Lemma 4.17) it follows that for at least $\gg_{L} |V|^2$ values of $v,v'$ we have $ P(v+v')-P(v)-  Q'(v')$ is of rank $\ll_{L} 1$, where  $Q'(v)(x) =  \tilde Q(v)(x, \ldots, x)/d!$  where $Q'(v)$ is linear in $v$. 
Recall now that $P(v+v')-P(v) - P(v')$ is of rank $\le L$, so that we get  a set $E$ of size  $\gg_{L} |V|$  of $v$ for which 
\[
P(v) =  Q'(v) + R(v)
\]
with $R(v)$ of rank $\le L$. Since $E \gg_{L} |V|$, by the Bogolyubov lemma (see e.g. \cite{wolf}))  $2E$ contains a subspace $E'$  of codimension $K \ll_{L} 1$ in $V$.  For $v \in E'$, define
\[
P'(v)=Q'(v).
\]
Let $P':V\to M_d$ be any extension of $P'$ linear in $v\in V$. Then $P'(v) - P(v)$ is of rank  $\ll_{L} 1$, and $P'(v)$ is a cocycle. 
\end{proof}
\begin{remark}\label{nonclassical}  In the case where $p\le d+2$, by the inverse theorem for the Gowers norms over finite fields the polynomial  $Q$ in the above argument on $V \times V^d$ would be replaced by a {\em nonclassical polynomial} see \cite{tz-low}, and the same argument would give that the approximate cohomology obstructions lie in the nonclassical degree $d$ polynomials - these are functions $P: V \to \mT$ satisfying $\Delta_{h_{d+1}} \ldots \Delta_{h_1}P \equiv 0$. 
\end{remark}

{\em Proof of  Theorem \ref{Main}}. Let $k=\mF _q$, and let $V$ be an countable vector space over $k$. Denote  $V_n=k^n$, then  $V=\cup V_n$.  Let $M\subset k[x_1,...x_n,...]$ be the subspace of homogeneous polynomials of degree $d$. For any $l\geq 1$ we denote by $N_l\subset M$ the subset of polynomials of in $x_1,...,x_l$ and denote by $p_l:M\to N_l$ the projection defined by $x_i\to 0$ for $ i> l$. Observe that  $p_n$ does not increase the rank. By Proposition \ref{main-finite}  there is a constant $C$ depending only on $L,d$ (and $k$) such that for any $n$ there exists a linear map $\phi _n:V_n\to M$ such that  rank $(P(v)-\phi _n (v))\le C$, for $v\in V_n$. 

We now show that the existence of such linear maps $\phi _n$ implies the existence of a linear map  $\psi :V\to M$ such that 
$rank (R(v)-\psi  (v))\leq C$, for $v\in V$.

\begin{lemma}\label{compatible} Let $X_n$, $ n\geq 0$ be finite not empty sets and  
 $f_n:X_{n+1}\to X_n$ be maps. Then one can find $x_n\in X_n$ such that $f_n(x_{n+1})=x_n$.
\end{lemma}

\begin{proof} This result is standard, but for the convenience of a reader we provide a proof.
The claim is obviously true if the maps $f_n$ are surjective.
For any $m>n$ we define the subset  $X_{m,n}\subset  X_n$ as the image of 
$$f_n\circ \ldots \circ f_{m-1}:X_m\to X_n$$ 
It is clear that for a fixed $n$ we have 
$$X_n\supset X_{n+1,n}\supset  \ldots \supset X_{m,n}\supset \ldots $$  
\ \\
We define $Y_n$ as the intersection $\cap _{m>n}X_{m,n}$. Since the set $X_n$ is finite, the sets $X_{m,n}$ stabilize as m grows and hence 
$Y_n$  is not empty. Let $\ti f_n$ be the restriction of $f_n$ on 
$Y_{n+1}$. Now the maps  $\ti f_n:Y_{n+1}\to Y_n$  are surjective, thus the lemma follows. 
\end{proof}

\begin{lemma}\label{lift}
Let $P:V\to M$ be a map such that for any $n$ there exists a linear map $\phi _n:V_n\to M$ such that  rank $(P(v)-\phi _n (v))\leq C$, $v\in V_n$. Then there exists  a linear map $\psi :V\to M$ such that  $rank (P(v)-\psi  (v))\leq C$, for all $v\in V$.
\end{lemma}

\begin{proof} 
 Let $l(n)$ be 
such that $P(V_n)\subset N_{l(n)}$  and $\psi _n=p_{l(n)}\circ \phi _n$. Since $p_n$ does not increase the rank, and since
$(p_{l(n)}\circ P)(V_n)= R(V_n)$ we have 
\[
(\star _n) \qquad rank(P(v)-\psi _n (v))\leq C,  \quad v\in V_n.
\]
We apply Lemma \ref{compatible} to the case when $X_n$ is the set of linear maps 
$\psi _n:V_n\to N_{l(n)}$ satisfying $(\star _n)$ and  
$f_n$ are the restriction from $V_{n+1}$ onto $V_n$
we find  the existence of  linear maps $\psi _n:V_n\to N_{l(n)}$ 
satisfying the condition $(\star _n)$ and such that the restriction of $\psi _{n+1}$ onto $V_n$ is equal to $\psi _n.$
The system $\{ \psi _n\}$ defines a linear map $\psi :V\to M$. 

\end{proof}

We now prove the result stated in the abstract by proving the equality  $R(V,r,k)=R^f(V,r,k)$. 
Let $\Gamma$ be an abelian group, $\Gamma =\cup \Gamma _n$ where $\Gamma_n$ are finitely generated groups. Let
$k$ be a finite field, and let $V,W$ be $k$-vector spaces with bases $v_j, w_j$. For $n\geq 1$ let $V_n, W_n$ be the spans of $v_j,w_j,1\leq j\leq n$. We denote by $i_n:V_n\to V_{n+1}$ the natural imbedding and by $\beta _n:W_{n+1}\to W_n$ the natural projection. Denote
$\Hom^f(V,W)$ the finite rank homomorphisms from $V \to W$. 

\begin{proposition}
Suppose  there exists $C=C(c)$ such that for  any map 
$a^f:\Gamma \to \Hom^f(V,W)$  such that  
$$r(a^f(\Gg '+\Gg '')-a^f(\Gg ')-a^f(\Gg ''))\leq c$$
there exists a homomorphism $\chi^f:\GG \to \Hom^f(V,W)$ such that $r(a^f(\Gg )-\chi ^f(\Gg ))\leq C$.
Then  for any map  $a:\GG \to \Hom(V,W)$ 
such that  
$$r(a(\Gg '+\Gg '')-a(\Gg ')-a(\Gg ''))\leq c$$
there exists a homomorphism $\chi :\Gamma \to \Hom(V,W)$ such that $r(a(\Gg )-\chi (\Gg ))\leq C$.
\end{proposition}

\begin{proof}
We will use the following fact that is an immediate consequence of K\"onig's lemma : Let $X$ be a locally finite tree, $x\in X$. If for any $N$ there exists a branch starting at $x$ of length $N$ then there exists an infinite branch starting at $x$. \\
\ \\
Let $a:\Gamma \to \Hom(V,W)$ be a map such that  
$$r(a(\Gg '+\Gg '')-a(\Gg ')-a(\Gg ''))\leq c$$
We define $F_n=\Hom (V_n,W_n)$. Let $Y_n=\Hom (\Gamma _n, F_n)$ and 
$q_n :Y_{n+1}\to Y_n$ be given by
$$q_n (\chi _{n+1})=\beta _n\circ \chi '_{n+1}\circ i_n$$
where $\chi '_{n+1}$ is the restriction of $\chi _{n+1}$ on $\GG _n$. \\
\ \\
We denote by $X_n\subset Y_n$ the subset of homomorphisms $\chi _n$ of $\Gamma _n$ such that 
$$r(\beta _n\circ a(\Gg )\circ i_n -\beta _n\circ \chi _n(\Gg )\circ i_n )\leq C.$$
Let $X$ be the disjoint union of $X_n$ and we connect 
$\chi _n\in X_n$ with $\chi _{n+1}\in X_{n+1}$ if 
$\chi _n=q_n(\chi _{n+1})$. \\
\ \\
By the assumption $X_n$ are finite not empty sets and for any $n$ there exists a branch from $X_0$ to $X_n$ (any $\chi _n\in X_n$ defines such a branch). Now the Lemma  \ref{compatible}  implies the 
existence a character $\chi :\Gamma \to \Hom(V,W)$ such that $r(a(\Gg )-\chi (\Gg ))\leq C$.
\end{proof}

To conclude the proof of Theorem \ref{Main} we calculate the kennel of the map $a^1$:

\begin{proposition}\label{kernel} The kernel of $a^1$ consists of maps $P$ of finite rank.
\end{proposition}

\begin{proof}
Suppose  $V, W$ are of dimension $n_1, n_2$ respectively. All the bounds below are independent of $n_1, n_2$. Suppose  $P:V \to M_d$ is a linear map with $r(P) \le L$. 
Let $\tilde P$ be the multilinear version of $P$  as in \eqref{multilinear}.
Let $f(v,\bar x) = \psi(\tilde P(v)(\bar x))$.
Now $\tilde P(v)(\bar x)$ is a multilinear polynomial on  $V\times W^d$ of degree $d+1$. \\
\ \\
For any fixed $v$ we have
\[ 
\mathbb E_{\bar x \in W^d} \psi(\tilde P(v)(\bar x)) \ge C(L),
\]
and thus
\[
\mathbb E_{v\in V}  \mathbb E_{\bar x \in W^d} \psi(\tilde P(v)(\bar x)) \ge C(L).
\]
It follows that  $\tilde P(v)(\bar x)$ is of bounded rank $\ll_L 1$ and thus of the form
\[
\tilde P(v)(\bar x) = \sum_{j=1}^K \tilde Q_j(v,\bar x) \tilde R_j(v, \bar x)
\]
with $\tilde Q_j, \tilde R_j$ of degree $\ge 1$ ,for any fixed $v$ also $\tilde Q_j, \tilde R_j$ are of degree $\ge 1$, and $K \ll_L 1$. For any fixed $x$, $\tilde P(V)$ is linear
and thus either $\tilde Q_j$ or $\tilde R_j$ are constant as a function of $v$. Recall that $P(v)(x) = \frac{1}{d!}\tilde P(v)(x,\ldots, x)$, and let 
$Q_j(v,x) =  \frac{1}{d!}\tilde Q_j(v,x,\ldots, x)$, similarly $R_j$. 
Let $J$ be the set of $j$ in the sum $P=\sum_j Q_jR_j$ such that $Q_j$ is linear in $v$ and does not depend on $x$. Let $V'=\bigcap _j \Ker Q_j$. The restriction of $P$ to $V'$ has finite (that is by a constant which does not depend on $n_1,n_2$ ) rank. Since 
$codim (V')\leq |J|$ we see that $P$ has rank that is bounded by a constant which does not depend on $n_1$ and $n_2$.\\
\ \\
Now let $V, W$ be infinite.  Let $V_n, l(n), p_{l(n)}$ be as in Lemma \ref{lift}.  Len $ P_n(V_n)= (p_{l(n)}\circ P)(V_n)$.  Now apply Lemma
\ref{compatible} for $X_n$ the collection of finite rank maps from $V_n \to N_{l(n)}$, and $f_n$ the restriction as before. This finishes a proof of Theorem \ref{Main}.
\end{proof}

\end{document}